\documentclass[article, 11pt, reqno]{amsart}

\usepackage{amssymb, amsmath}
\usepackage[margin=2.5cm]{geometry}
\usepackage{hyperref}
\usepackage[dvipsnames]{xcolor}
\usepackage[color=Purple!50!white, textwidth=20mm]{todonotes}
\newtheorem{theorem}{Theorem}

\theoremstyle{definition}

\numberwithin{equation}{section}

\begin{document}

\title{Ramsey numbers of trails and circuits}

\author{David Conlon}
\address{Department of Mathematics, California Institute of Technology, Pasadena, CA 91125, USA}
%\curraddr{}
\email{dconlon@caltech.edu}
\thanks{ Conlon was supported by NSF Award DMS-2054452 and Tyomkyn by ERC Synergy Grant DYNASNET 810115, the H2020-MSCA-RISE Project CoSP- GA No.~823748 and GA\v{C}R Grant 19-04113Y}

\author{Mykhaylo Tyomkyn}
\address{Department of Applied Mathematics, Charles University, 11800 Prague, Czech Republic}
%\curraddr{}
\email{tyomkyn@kam.mff.cuni.cz}
\thanks{}

\date{}

\maketitle

\begin{abstract}
We show that every two-colouring of the edges of the complete graph $K_n$ contains a monochromatic  
trail or circuit of length at least $2n^2/9 +o(n^2)$, which is asymptotically best possible.
\end{abstract}

A \emph{trail} in a graph is a walk without repeated edges and a \emph{circuit} is a closed trail, with the same first and last vertex. The \emph{length} of a trail or circuit is its number of edges. Recently, Osumi~\cite{Osu} investigated Ramsey numbers for trails, proving that every two-colouring of the edges of $K_n$ contains a monochromatic trail of length at least $n-1$, while there are two-colourings where the longest monochromatic trail has length at most $n^2/4+o(n^2)$. In this note, we improve these results. 

\begin{theorem} \label{thm:main}
Every two-colouring of the edges of $K_n$ contains a monochromatic circuit (and so a trail) of length at least $2n^2/9+O(n^{3/2})$ and this is asymptotically tight. 
\end{theorem}

\begin{proof}
For the upper bound, consider the red/blue colouring of $K_n$ where the red edges form a complete bipartite graph between two blue cliques of orders $n/3$ and $2n/3$, each rounded appropriately. It is easily checked that the largest monochromatic component has size $2n^2/9+O(n)$ and so the longest monochromatic trail or circuit has length at most $2n^2/9+O(n)$. 

For the lower bound, suppose that we are given a red/blue colouring of the complete graph $K_n$. After removing, for each colour class, a suitable forest that meets all odd degree vertices (see, for instance,~\cite[Proposition 2.1]{Sze}), we may assume that each colour class is Eulerian, in the sense that every vertex has even degree in both red and blue. However, this is not immediately helpful, since the colour classes may be disconnected.  

Suppose that the largest blue component
$U_1$ has order $n_1$, 
noting that the bipartite graph between $U_1$ and its complement $U_1^c$ is, apart from the at most $2n$ missing edges, complete in red. We claim that if $n_1 \leq n-2\sqrt{n}$, then the red bipartite graph between $U_1$ and $U_1^c$ has a connected component which includes all but $\sqrt{n}$ vertices of $U_1$. To see this, note that there are at least $n_1(n-n_1) - 2n$ red edges between $U_1$ and $U_1^c$, so there is a vertex in $U_1^c$ with degree at least 
$n_1 - \frac{2n}{n-n_1} \geq n_1 - \sqrt{n}$
in $U_1$. Therefore, all of these at least $n_1 - \sqrt{n}$ vertices, which we label 
$V_1$, lie in a common red component, as required. Moreover, this component contains at least $(n_1 - \sqrt{n})(n-n_1) - 2n$ edges.

If now $n_1 < n/3$, then all blue components have order less than $n/3$, so the number of red edges is at least $(1/2)\cdot n \cdot 2n/3 - 2n = n^2/3 + O(n)$ and the average red degree is at least $2n/3 + O(1)$. After deleting a bounded number of vertices, we may also assume that every vertex has red degree at least $n/2$, which implies that the remaining graph is connected. By~\cite[Theorem 1.3]{Sze}, which says  that any connected graph with average degree $t$ contains a trail of length $\binom{t}{2} + O(t)$, this then implies that there is a red trail of length at least $2n^2/9 + O(n)$, as required. Since the component containing this trail is Eulerian, we also have a red circuit of at least the same length.

If $n/3 \leq n_1 \leq 2n/3$, then the red component containing $V_1$ has size at least 
\[(n_1 - \sqrt{n})(n-n_1) - 2n \geq (n/3 - \sqrt{n})2n/3 - 2n = 2n^2/9 + O(n^{3/2}).\]
But this component is Eulerian, so we have a circuit through all of the edges of the component, giving the required circuit (and trail) of length at least $2n^2/9 + O(n^{3/2})$. 

If, instead, $2n/3 < n_1 \leq n - 2\sqrt{n}$, consider the induced graph on $U_1$. Ignoring colours for now, this graph has at least $\binom{n_1}{2} - 2n$ 
edges. Moreover, since $|V_1|\geq n_1-\sqrt{n}$, all but $\binom{\sqrt{n}}{2} \leq n$ of these edges are incident with a vertex in $V_1$, so that the number of edges in $U_1$ incident with a vertex in $V_1$ is at least $\binom{n_1}{2} +O(n)$. Together with the edges in the bipartite graph between $V_1$ and $U_1^c$, we have at least
\[\binom{n_1}{2} + n_1(n-n_1) +O(n^{3/2}) \geq 4n^2/9 + O(n^{3/2})\]
edges. Therefore, either there are at least $2n^2/9 + O(n^{3/2})$ edges in the red component containing $V_1$, which again completes the proof, or there are at least $2n^2/9 + O(n^{3/2})$ edges in the blue component in $U_1$. Since this component is also Eulerian, this again completes the proof.

It remains to deal with the case where $n_1 > n - 2\sqrt{n}$. By symmetry, we may also assume that the largest red component has more than $n - 2\sqrt{n}$ vertices. The number of edges which are not contained within the intersection of the vertex sets of these  components is at most $4\sqrt{n} \cdot n = 4 n^{3/2}$, so the total number of edges in the intersection is at least \[\binom{n}{2} - 4 n^{3/2} - 2n = \frac{n^2}{2} + O(n^{3/2}).\]
Therefore, either the largest red or the largest blue component, both of which are again Eulerian, contains at least $n^2/4 + O(n^{3/2})$ edges, more than required.
\end{proof}

Inverting the statement of Theorem~\ref{thm:main}, we see that the Ramsey number of a trail with $\ell$ edges, that is, the smallest $n$ such that every two-colouring of the edges of $K_n$ contains a monochromatic trail with $\ell$ edges, is $3 \sqrt{\ell/2} + o(\sqrt{\ell})$. Similarly, the Ramsey number for the family of all circuits with at least $\ell$ edges is $3 \sqrt{\ell/2} + o(\sqrt{\ell})$. It remains an interesting question to determine the Ramsey number for circuits of a given fixed length.

It would also be interesting to investigate the analogue of Theorem~\ref{thm:main} for more than two colours. It is reasonably easy to see that for every natural number $k$, there exists $c_k$ such that every $k$-colouring of the edges of $K_n$ contains a monochromatic circuit of length at least $c_k n^2 + o(n^2)$. Indeed, if we again delete a forest for each colour class, we have a graph with $\binom{n}{2} - kn$ edges where every coloured component is Eulerian. For $n$ sufficiently large, one of the colours in this graph has average degree at least $n/k+o(n)$, which, by a standard folklore result, implies that this colour has a subgraph of minimum degree $n/2k+o(n)$. But then there is a component in this colour with at least $n^2/8k^2 + o(n^2)$ edges. We note that this is also close to sharp. To see this, we note that  when $k-1$ is a prime power and $n$ is a multiple of $(k-1)^2$, there is a construction of Gy\'arf\'as (see, for example,~\cite{Gya}) using affine planes where every monochromatic component has order at most $n/(k-1)$ and, hence, at most $\binom{n/(k-1)}{2}$ edges (or, with a more careful analysis, $n^2/2k(k-1) + O(n)$ edges). 
The natural next step would be to determine the best possible constant $c_k$. Clearly, this problem is closely related to the question of determining the largest number of edges in a monochromatic component in any $k$-colouring of the edges of $K_n$. In fact, the answer should be asymptotically the same in both cases. Our arguments verify this for $k=2$. For the next case, $k = 3$, we suspect that the bound coming from Gy\'arf\'as' construction, $n^2/12 + O(n)$ edges, is correct.

\end{document}